\documentclass[12pt,a4paper]{amsart}
\usepackage[hmarginratio=1:1, right=2.5cm]{geometry}
\usepackage{amsmath}
\usepackage{amssymb}
\usepackage{enumerate}
\usepackage{color}
\usepackage[usenames,dvipsnames,svgnames,table]{xcolor}
\usepackage{tikz}
\usetikzlibrary{arrows.meta, decorations.markings}

\makeatletter
\@namedef{subjclassname@2010}{%
	\textup{2010} Mathematics Subject Classification}
\makeatother

\newtheorem{thm}{Theorem}[section]

\newtheorem{lem}[thm]{Lemma}

\theoremstyle{definition}
\newtheorem{defin}[thm]{Definition}
\newtheorem{rem}[thm]{Remark}

\numberwithin{equation}{section}
\usepackage{hyperref}
\hypersetup{
	colorlinks=true,
	linkcolor=blue,
	filecolor=magenta,      
	urlcolor=cyan,
	citecolor=blue, 
}

\tikzset{
	mid arrow/.style={
		postaction={
			decorate,
			decoration={
				markings,
				mark=at position 0.5 with {\arrow{Stealth[scale=1.2]}}
			}
		}
	}
}
\begin{document}
	
	\title[On the counting function of square-full numbers ]{On the counting function of square-full numbers  }
	\author{Ayyadurai Sankaranarayanan}
	\address{Ayyadurai Sankaranarayanan, Kerala School of Mathematics, Kozhikode, Kerala-673571, India}
	\email{sank@ksom.res.in  }
	
	\author{P. Akhilesh}
	\address{P. Akhilesh, Kerala School of Mathematics, Kozhikode, Kerala-673571, India}
	\email{akhi@ksom.res.in }

	\begin{abstract} We  study the distribution of square-full numbers under the assumptions of Riemann hypothesis and that the zeros of $\zeta(s)$ are simple. Thus the error term is improved to a considerable extent with an extra main term.
		
	\end{abstract}
	

	\maketitle
	
	\begin{section}{Introduction}
		\begin{defin}{(k-full numbers)}
			Let $k\geqslant 2$, be an integer. A natural number $n$ is said to be a $k$-full number or powerful numbers if in its cannonical representation $n=p_1^{a_1}\ldots p_r^{a_r}$, each $a_i \ge k$.
			
		\end{defin}
		Let us define : For any integer $k \ge 2$,
		\begin{equation*}
			G(k) := \{ n \in {\mathbb N} : p|n \implies p^k | n \}  
		\end{equation*}
		and the characteristic function
		\begin{equation*}
			f_k(n) = 1 \ {\rm if} \ n \in G(k) \ ; \ 0  \ {\rm if} \ n \notin G(k).
		\end{equation*}
		A classical question in Analytic Number Theory is to study the error term $E_2(x)$ of 
		\begin{equation}\label{E2}
			B_2(x):=	\sum_{n\leqslant x} f_2(n)= \frac {\zeta (\frac {3}{2})}{\zeta (3)} x^{\frac {1}{2}} + \frac {\zeta (\frac {2}{3})}{\zeta (2)}  x^{\frac {1}{3}}+E_2(x).
		\end{equation}
		Note that $B_2(x)$ counts the number of square full numbers up to $x$.
		
		Unconditionally one knows that $E_2(x)\ll x^{\frac{1}{6}}{\rm exp}(-c({\log x})^{\frac{3}{5}}(\log \log x)^{-\frac{1}{5}})$ (see \cite{IV}) . There is a strong belief that $E_2(x)=\Omega(x^{\frac{1}{12}-\epsilon})$ and also even  $E_2(x)=\Omega_{\pm}(x^{\frac{1}{12}-\epsilon})$  hold for any $\epsilon>0$.  
		
		Thus finding a tight upper bound for $|E_2(x)|$  becomes more interesting even under certain unproved hypothesis.
		
		\noindent
		{\bf Riemann Hypothesis (RH):} All the non-trivial zeros of $\zeta(s)$ lie on the critical line ${\Re}(s)=\frac{1}{2}$
		\\{\bf Strong Riemann Hypothesis :}  RH  plus each non-trivial zero is simple.
		\vskip 2mm
		\noindent
		If we write, for any $\epsilon > 0$,
		\begin{equation*}
			E_2(x) \ll x^{\theta + \epsilon}
		\end{equation*}
		then, assuming RH, D. Suryanarayana and K. Sitaramachandra Rao (see \cite{SS}) proved that $\theta = \frac {13}{81} = 0.1604 \cdots$ and it is improved to $\theta = \frac {121}{860} = 0.1406 \cdots$
		a few years ago by H.Q. Liu  assuming RH (\cite{HQL}). There were several improved results on the value of $\theta$ under RH in between these two works and the readers may refer to \cite{HQL}.
		
		Throughout the paper, $s=\sigma + it$ and $\epsilon > 0$ is any small constant. The letters $c, c_1, c_2, \cdots$ and $c_1^{\prime}, c_2^{\prime}, \cdots$ denote effective positive constants, need not be the same at each occurrence.
		
		Throughout the paper we assume Strong Riemann Hypothesis (unless stated otherwise). We prove :
		\begin{thm}\label{T1}
			We assume the strong Riemann hypothesis. For any $x\geqslant x_0$, where $x_0$ is sufficiently large, there are effective complex constant $c_{\{\gamma /6 \}}$ such that the asymptotic formula
			\begin{equation}\label{E5}
				\sum_{n\leqslant x} f_2(n)= \frac {\zeta (\frac {3}{2})}{\zeta (3)} x^{\frac {1}{2}} + \frac {\zeta (\frac {2}{3})}{\zeta (2)}  x^{\frac {1}{3}}+	
				\sum_{\substack{\rho=\frac{1}{12}+i\frac{\gamma}{6},\\|\gamma |< c_1 x^{\frac {1}{2}},\\\zeta(6\rho)=0}}c_{\frac{\gamma}{6}} \ x^{\frac{1}{12}+i\frac{\gamma}{6}}
				+O_{\epsilon} \left( x^{10 \epsilon} \right)
			\end{equation}
			holds. We also have
			\begin{equation}\label{E5a}
				\sum_{n\leqslant x} f_2(n)= \frac {\zeta (\frac {3}{2})}{\zeta (3)} x^{\frac {1}{2}} + \frac {\zeta (\frac {2}{3})}{\zeta (2)}  x^{\frac {1}{3}}+	
				\sum_{\substack{\rho=\frac{1}{12}+i\frac{\gamma}{6},\\|\gamma |< c_2 x^{\frac {5}{12}},\\\zeta(6\rho)=0}}c_{\frac{\gamma}{6}} \ x^{\frac{1}{12}+i\frac{\gamma}{6}}
				+O_{\epsilon} \left( x^{\frac {1}{12} + 10 \epsilon} \right).
			\end{equation}	
		\end{thm}
		
		\begin{rem}
			The asymptotic formula (\ref{E5}) concentrates to get only the best possible  error term. We should also notice that the length of the summand in the main term is compromised  to certain extent. However the asymptotic formula (\ref{E5a}) concentrates to obtain an expected O-term with the length of the summand in the main term is optimal. The constants $c_1$ and $c_2$ are positive (may be small) but fixed. Of course the most interesting and the difficult part here is to determine the size in absolute value of the extra main term 
			\begin{equation}
				\sum_{\substack{\rho=\frac{1}{12}+i\frac{\gamma}{6},\\|\gamma |< c_2 x^{\frac{5}{12}},\\\zeta(6\rho)=0}}c_{\frac{\gamma}{6}} \ x^{\frac{1}{12}+i\frac{\gamma}{6}}.
			\end{equation}
			This seems to be a highly challenging task.
		\end{rem}
		If we wish to study unconditionally, then we notice certain facts about the Riemann zeta function. We know that there are infinitely many (in fact quantitative version ) zeros of $\zeta(s)$ on the critical line itself. Some zeros may occur with multiplicity more than one and it may be of order even at most $\ll {\rm log}T$. Taking in to account all these possibilities (including cancellations of the roots of $\zeta(2s)$, $\zeta (3s)$ and $\zeta(6s)$ if any ), we can prove:
		
		\begin{thm}\label{T2}
			For $x \geqslant x_0$, where $x_0$ is sufficiently large, unconditionally we have
			
			\begin{equation}\label{E5b}
				\sum_{n\leqslant x} f_2(n)= \frac {\zeta (\frac {3}{2})}{\zeta (3)} x^{\frac {1}{2}} + \frac {\zeta (\frac {2}{3})}{\zeta (2)}  x^{\frac {1}{3}}+
				\sum_{\substack{\rho=\frac{\beta}{6}+i\frac{\gamma}{6},\; \zeta(6\rho)=0,\\0<6\beta <1,\;|\gamma |< c_3 x^{\frac{105}{226-168\epsilon}},\\1\leqslant {\rm ord}(6\rho)\ll \log (c_3 x^{\frac{105}{226-168\epsilon}})}}{{\rm Res}\atop s=\rho}\left[\frac{\zeta(2s) \ \zeta (3s)}{\zeta(6s)}\; \frac{x^s}{s}\right]+O_{\epsilon}\left ( x^{\frac{1}{12} + 10 \epsilon} \right)
			\end{equation}
			for some positive constant $c_3$.
		\end{thm}
		
		\begin{rem}
			1. It should be noted that in estimating the contributions coming from the left vertical line and the horizontal lines, we need to use only the unconditional bounds (see Lemmas \ref{RS} to \ref{UL}). Hence the proof of the theorem \ref{T2} is almost verbatim  the same as of the proof of Theorem \ref{T1} and thus we only sketch it's proof.\\
			2. We also discuss about an implication of the Lindlof Hypothesis in the general case of $k-$ full numbers in the last section and this improves the known upper bound for $|E_k(x)|$ under the assmption of the Lindelof Hypothesis (LH)  for all integers $k \ge 4$.\\
			3. For some effective and non-effective results on certain arithmetical functions, we refer to \cite{BR}.\\
			4. From the proof we observe that the optimal value of $T$ is $T = c_3 x^{\frac {5}{4(3-2\kappa)}}$ where $\kappa > 0$ is the exponent coming from the growth estimate : $|\zeta (\frac {1}{2}+it)| \ll (|t|+10)^{\kappa} $.  By convexity principle, we know that $\kappa = \frac {1}{4}+\epsilon$. This estimate is unconditional for any $\kappa \ge \frac {13}{84}+\epsilon$ (as of now) (see \cite{B}, \cite{TE}). The Lindel\"of hypothesis asserts that this estimate holds good with $\kappa = \epsilon$ for any $\epsilon > 0$.
			In the proof of the theoem \ref{T2} below we have used the unconditional value  $\kappa = \frac {13}{84}+\epsilon$ (see \cite{B}).\\			
			5. {\bf A plausible Conjecture :} {\it For any} $\kappa > 0$, {\it the estimate}
			\begin{equation}\label{E5c}
				\sum_{\substack{\rho=\frac{\beta}{6}+i\frac{\gamma}{6},\; \zeta(6\rho)=0,\\0<6\beta <1,\;|\gamma |< c_3 x^{\frac {5}{4(3-2\kappa)}},\\1\leqslant {\rm ord}(6\rho)\ll \log (c_3 x^{\frac {5}{4(3-2\kappa)}})}}{{\rm Res}\atop s=\rho}\left[\frac{\zeta(2s) \ \zeta (3s)}{\zeta(6s)}\; \frac{x^s}{s}\right] = O_{\epsilon}\left ( x^{\frac{1}{12} + 10 \epsilon} \right).
			\end{equation}
			{\it holds for any $\epsilon > 0$}.
			
			\noindent
			This may very well happen {\it but we have no idea to substantiate \ref{E5c}.} \\
		\end{rem}

	\end{section}

	\begin{section}{preliminaries and some lemmas}
		\begin{defin}{Riemann zeta function}
			\begin{equation}\label{E6}
				\zeta(s) = \sum_{n=1}^{\infty} \frac{1}{n^s}, \quad {\Re}(s) > 1
			\end{equation}
		\end{defin}
		The Riemann zeta function can be analytically continued to $\mathbb{C}\setminus \{1\}$ by using the functional equation.
		\begin{equation}\label{E18}
			\zeta(s) = \chi(s) \zeta(1-s)
		\end{equation}
		were the conversion factor  $\chi$ has the growth rate of
		\begin{equation}\label{E19}
			\left|	\chi(\sigma + it) \right|\asymp |t|^{\frac{1}{2} -\sigma}
		\end{equation}
		
		\subsection{Some Lemmas}
		\begin{lem}
			Let $k \ge 2$ be any integer.
			For $\Re (s) > \frac {1}{k}$, we have
			\begin{equation}\label{E7}
				F_k(s) := \sum_{n=1}^{\infty}\frac{f_k(n)}{n^s}=\prod \limits_{p} \left ( 1 + \frac {p^{-ks}}{1-p^{-s}} \right ).
			\end{equation}
			In particular,
			\begin{equation}\label{E7a}
				F_2(s) := \sum_{n=1}^{\infty}\frac{f_2(n)}{n^s}= \frac {\zeta (2s) \zeta (3s)}{\zeta (6s)}.
			\end{equation}	
		\end{lem}
		Proof:-
		By using the Euler product  for ${\Re}(s)>1/k$, we have
		\begin{equation*} \label{E7b}
			F_k(s) := \sum_{n=1}^{\infty}\frac{f_k(n)}{n^s}= \prod \limits_{p} \left ( 1 + p^{-ks} + p^{-(k+1)s} + \cdots \right ) = \prod \limits_{p} \left ( 1 + \frac {p^{-ks}}{1-p^{-s}} \right ).
		\end{equation*}
		In particular, we note that :
		\begin{equation}\label{E7c}
			F_2(s) := \sum_{n=1}^{\infty}\frac{f_2(n)}{n^s}=\prod \limits_{p} \left ( 1 + \frac {p^{-2s}}{1-p^{-s}} \right ) = \prod \limits_{p} \left (  \frac {1-p^{-s}+p^{-2s}}{1-p^{-s}} \right ).
		\end{equation}
		We observe that if $|w| < 1$, then
		
		\begin{equation*}
			\frac {1-w^6}{(1-w^2)(1-w^3)} = \frac {1+w^3}{1-w^2} = \frac {(1+w)(1-w+w^2)}{(1+w)(1-w)}
			= \frac {1-w+w^2}{1-w}.
		\end{equation*}	
		Hence, we obtain (for $\Re s > \frac {1}{2}$, with $w=p^{-s}$)
		
		\begin{equation*}
			F_2(s) = \frac {\zeta (2s) \zeta (3s)}{\zeta (6s)}.
		\end{equation*}
		This proves the lemma.
		
		\begin{lem}{\cite{RS}}\label{RS}
			For $\frac{1}{2}\leqslant \sigma \leqslant 2$, $T$- sufficiently large, there exists a $T^{\ast}\in [T,T+T^{\frac{1}{3}}]$ such that  the bound
			\begin{equation}\label{E8}
				\log (\zeta(\sigma\pm it))\ll(\log \log (T^{\ast}))^2 \ll (\log \log (T))^2
			\end{equation}
			holds uniformly for $\frac{1}{2}\leqslant \sigma \leqslant 2$ and thus we have
			\begin{equation}\label{E9}
				|\zeta(\sigma\pm i T^{\ast})|^{\pm1}\,\,\ll\exp \left (c (\log \log T)^2 \right)
			\end{equation}
			on the horizondal line  $t=T^{\ast}$ uniformly for $\frac{1}{2}\leqslant \sigma \leqslant 2$. This implies
			\begin{equation}\label{E9a}
				\frac{1}{|\zeta(\sigma\pm i T^{\ast})|}\ll \exp \left (c (\log \log T)^2 \right)
			\end{equation}
		\end{lem}
		\begin{lem}{\cite{IV}(Vinogradov)}\label{VL}
			For $|t|\geqslant 10$, there exists an absolute constant $C$ such that in the region 
			\begin{equation*}\label{E9b}
				\sigma \geqslant 1-\frac{C}{(\log |t|)^{\frac{2}{3}}(\log \log |t|)^{\frac{1}{3}}}
			\end{equation*}
			we have the bounds :
			\begin{equation}\label{E10a}
				\zeta(\sigma+it)\ll \left (\log |t| \right)^{\frac{2}{3}+\epsilon}
			\end{equation}
			
			\begin{equation}\label{E10}
				\frac{1}{\zeta(\sigma+it)}\ll \left (\log |t| \right)^{\frac{2}{3}+\epsilon}
			\end{equation}
			for any $\epsilon>0$.
		\end{lem}
		\begin{lem}\label{UL} For $|t| \ge t_0$ where $t_0$ is sufficiently large, we have (with $\kappa \ge \frac {13}{84}+\epsilon$ unconditionally)
			\begin{equation}
				\zeta(\sigma+it)\ll |t|^{2 \kappa (1-\sigma)} \log |t|
			\end{equation}
			uniformly for $\quad \frac{1}{2}\leqslant \sigma\leqslant 1+\frac{1}{\log t}$.
		\end{lem}
		\begin{proof}
			We notice that $\zeta(\frac{1}{2}+it)\ll t^{\kappa}$ and $\zeta(1+\frac{1}{\log t}+it)\ll(\log t)^{\frac{2}{3}+\epsilon}$ (see \cite{B} \cite{TE}). Applying maximum modulus principle to the function
			\begin{equation}
				F(w)=\zeta(w) X^{w-s}e^{(w-s)^2}
			\end{equation}
			in a suitable rectangle, we obtain
			\begin{equation}
				\zeta(\sigma+it)\ll |t|^{2 \kappa (1-\sigma)} \log |t|
			\end{equation}
			uniformly for  $\frac{1}{2} \leqslant \sigma \leqslant 1+ \frac{1}{\log t}$. 
		\end{proof}

	\end{section}
	
	\begin{section}{Proofs of the Theorems}
		By Perron's formula, we have
		\begin{equation}\label{E11}
			\sum_{n\leq x} f_2(n) = \frac{1}{2\pi i}\int\limits_{\frac {1}{2}+\frac{1}{{\log x}}-iT}^{\frac {1}{2}+\frac{1}{{\log x}}+iT} \frac {\zeta (2s) \zeta (3s)}{\zeta (6s)} \frac{x^s}{s}ds+O\left(\frac{x^{\frac {1}{2} +\frac{1}{\log x}}(\log x)^2}{T}\right)
		\end{equation}
	where $10\leqslant T\leqslant x$. We choose the open parameter $T$ such that  $T= \pm 6 T^{\ast} $ such that $\frac{U}{12} <T^{\ast}<\frac{U}{6}$ to satisfy  \mbox{Lemma \ref{RS}} so that we have \[ \frac{1}{\zeta(6\sigma+i6T^{\ast})}\ll {\rm exp}\{c (\log \log U)^2\}\ll {\rm exp}\{c (\log \log T)^2\}.\] More importantly throughout the proof, we use  \mbox{Lemma \ref{RS}} for the denominator function namely to $\frac{1}{\zeta(6\sigma+i6t)}$ suitably.

		We move the line of integration to ${\Re}(s)= \sigma_0$ where $-\frac {1}{2} < \sigma_0 < 0$.
	The quantity $\sigma_0$ is kept open and will be chosen at the end suitably.
	
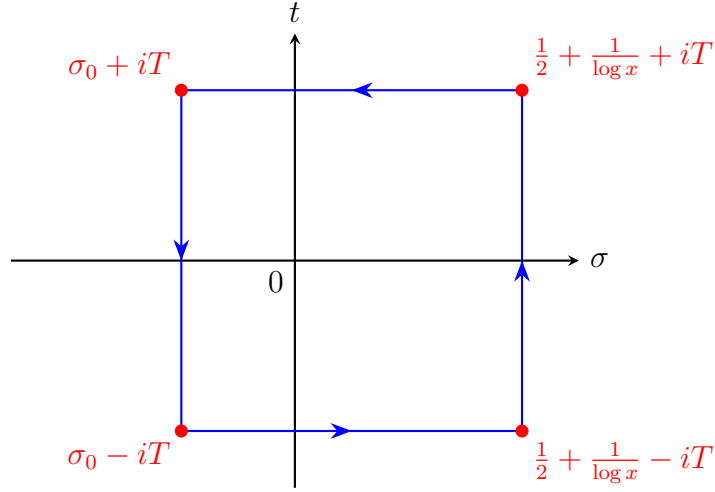
\begin{center}
		\begin{tikzpicture}[>=stealth, scale=1.5]
			\draw[->, thick] (-2.5, 0) -- (2.5, 0) node[right] {$\sigma$};
			\draw[->, thick] (0, -2) -- (0, 2) node[above] {$t$};
			
			\node[below left] at (0,0) {$0$};
			
			\draw[blue, thick, mid arrow] (-1, -1.5) -- (2, -1.5); 
			\draw[blue, thick, mid arrow] (2, -1.5) -- (2, 1.5);   
			\draw[blue, thick, mid arrow] (2, 1.5) -- (-1, 1.5);   
			\draw[blue, thick, mid arrow] (-1, 1.5) -- (-1, -1.5); 

			\filldraw[red] (-1,  1.5) circle (1.5pt) node[above left]  {$\sigma_0+iT$};
			\filldraw[red] ( 2,  1.5) circle (1.5pt) node[above right] {$\frac {1}{2}+\frac{1}{\log x}+iT$};
			\filldraw[red] (-1, -1.5) circle (1.5pt) node[below left]  {$\sigma_0-iT$};
			\filldraw[red] ( 2, -1.5) circle (1.5pt) node[below right] {$\frac {1}{2}+\frac{1}{\log x}-iT$};
		\end{tikzpicture}
	\end{center}

		We observe that in the rectangle formed by the line segments joining the vertices  
		$\frac {1}{2}+\frac{1}{\log x}-iT$,
		$\frac {1}{2}+\frac{1}{\log x}+iT$,
		$\sigma_0+iT$,
		$\sigma_0-iT$,
		taken in the anti-clock wise direction in this order,  there are simple poles at $s=\frac {1}{2}$,
		$\frac {1}{3}$, $s=0$ and $s=\rho=\frac{1}{12}+i\frac{\gamma}{6}$ such that $\zeta(6\rho)=0$ of the integrand $\frac{\zeta(2s) \zeta (3s)}{\zeta(6s)}\frac{x^s}{s}$.
		
		So by Cauchy's residue theorem, the residues are :
		\begin{equation}\label{E12}
			\buildrel \text{\normalsize Res} \over{_{s=\frac {1}{2}}} \left[\frac{\zeta(2s) \zeta (3s)}{\zeta(6s)}\frac{x^s}{s}\right]= \frac {\zeta (\frac {3}{2})}{\zeta (3)} x^{\frac {1}{2}}
		\end{equation}
		
		\begin{equation}\label{E12a}
			\buildrel \text{\normalsize Res} \over {_{s=\frac {1}{3}}} \left[\frac{\zeta(2s) \zeta (3s)}{\zeta(6s)}\frac{x^s}{s}\right]= \frac {\zeta (\frac {2}{3})}{\zeta (2)} x^{\frac {1}{3}}
		\end{equation}

		\begin{equation}\label{E13}
			\buildrel \text{\normalsize Res} \over {_{s=0}} \left[\frac{\zeta(2s) \zeta (3s)}{\zeta(6s)}\frac{x^s}{s}\right]= \zeta (0) = - \frac {1}{2}
		\end{equation}

		\begin{equation}\label{E14}
			\buildrel \text{\normalsize Res} \over {_{ s=\rho={\frac{1}{12}+i\frac{\gamma}{6}}}}\left[\frac{\zeta(2s) \zeta (3s)}{\zeta(6s)}\frac{x^s}{s}\right]=c_{ \frac {\gamma}{6}} \ x^{\frac{1}{12}+i\frac{\gamma}{6}}
		\end{equation}
		from our assumption ( strong Riemann hypothesis). Depending upon the choice of  $\sigma_0$ in the said range, there may be a simple pole at $s=-\frac {1}{3}$ coming from the trivial zero at $s=-\frac {1}{3}$ of the denominator factor $\zeta (6s)$. If this pole is to be accounted then, 
		
		\begin{equation}\label{E13}
			\buildrel \text{\normalsize Res} \over {_{s=-\frac {1}{3}}} \left[\frac{\zeta(2s) \zeta (3s)}{\zeta(6s)}\frac{x^s}{s}\right] \ll x^{-\frac {1}{3}} \ll 1
		\end{equation}		
		which is anyway much smaller in order than  the order of the error what we want to prove.
		
		Thus we obtain
		\begin{eqnarray}\label{E15}
			\sum_{n\leqslant x} f_2(n) &&= \frac {\zeta (\frac {3}{2})}{\zeta (3)} x^{\frac {1}{2}} +
			\frac {\zeta (\frac {2}{3})}{\zeta (2)} x^{\frac {1}{3}} \\
			\nonumber	 &&	+\sum_{\substack{\rho =\frac{1}{12}+i\frac{\gamma}{6},\\|\gamma|<T,\\ \zeta(6\rho)=0}}c_{\frac {\gamma}{6}} \ x^{\frac{1}{12}+i{\frac{\gamma}{6}}}+O(E_1)+O(E_2)+ O(E_3) + O(1)
		\end{eqnarray}
		where $E_1$ is the contribution coming from the horizontal lines and  $E_2$ is the left vertical line  contribution and $E_3$ is the error term coming from the tail portion of  the Perron's formula.

		\subsection{Left vertical line contribution in absolute value ($E_2$)}\vspace{2em}\mbox{}\\
		Note that
		\begin{eqnarray}\label{E16}
			\nonumber \left|\frac{1}{2\pi i}\int\limits_{\sigma_0-iT}^{\sigma_0+iT}
			\frac{\zeta(2s) \zeta (3s)}{\zeta(6s)}\frac{x^s}{s}ds \right|&\leqslant  \int_{\sigma_0-iT}^{\sigma_0+iT} \left|\frac{\zeta(2s) \zeta (3s)}{\zeta(6s)}\right| \frac{x^{\sigma0}}{\left|\sigma_0+it\right|}dt
		\end{eqnarray}
		For $|t| \leqslant 10 $, since $-\frac {1}{2} < \sigma_0 < 0$, using functional equation, we find that $\left|\frac{\zeta(2s) \zeta (3s)}{\zeta(6s)}\right|\ll 1$ so that
		\begin{equation*}\label{E17}
			\frac {1}{2\pi} \int_{\sigma_0-i10}^{\sigma_0+i10} \left|\frac{\zeta(2s) \zeta (3s)}{\zeta(6s)}\right| \frac{x^{\sigma0}}{\left|\sigma_0+it\right|}dt \ll \frac {x^{\sigma_0}}{\sigma_0} \ll x^{\sigma_0}.	
		\end{equation*}
		We have the functional equation of $\zeta(s)$ in the form
		\begin{equation}\label{E18}
			\zeta(s) = \chi(s) \zeta(1-s)
		\end{equation}
		with conversion factor  $\chi$, has the growth rate
		\begin{equation}\label{E19}
			| \chi(\sigma + it) | \asymp |t|^{\frac{1}{2} -\sigma} \quad \text{for}\;\; t\geqslant 10 \quad.
		\end{equation}
		We observe that (for $|t| \ge 10$)
		\begin{eqnarray}\label{E20}
			\nonumber &&\left|\frac{\zeta(2\sigma_0+2it) \zeta (3\sigma_0 +3it)}{\zeta(6\sigma_0+6it)}\right|=	\left|\frac{\chi(2\sigma_0+2it) \chi (3\sigma_0+3it)}{\chi(6\sigma_0+6it)}\right|	\left|\frac{\zeta(1-2\sigma_0-2it) \zeta (1-3\sigma_0-3it)}{\zeta(1-6\sigma_0-6it)}\right|\\
			\nonumber&&\asymp \left|\frac{t^{\frac{1}{2}-2\sigma_0} t^{\frac {1}{2}-3\sigma_0}}{t^{\frac{1}{2}-6\sigma_0}}\right|	
			\left|\frac{\zeta(1-2\sigma_0-2it) \zeta (1-3\sigma_0-3it)}{\zeta(1-6\sigma_0-6it)}\right|\\ 
			\nonumber&&\asymp |t|^{\frac {1}{2}+\sigma_0}	
			\left|\frac{\zeta(1-2\sigma_0-2it) \zeta (1-3\sigma_0-3it)}{\zeta(1-6\sigma_0-6it)}\right|
		\end{eqnarray}
		\subsubsection{Bound for $\left|\frac{\zeta(2s) \zeta (3s)}{\zeta(6s)}\right|$ for $|t| \ge 10$}\mbox{}\\
		
		For  $-\frac {1}{2}<\sigma_0  \le - \epsilon < 0$ then we have 
		\begin{eqnarray}\label{E20a}
			\nonumber \left|\frac{\zeta(2\sigma_0+2it) \zeta (3\sigma_0 +3it)}{\zeta(6\sigma_0+6it)}\right|
			&&\asymp |t|^{\frac {1}{2}+\sigma_0}	
			\left|\frac{\zeta(1-2\sigma_0-2it) \zeta (1-3\sigma_0-3it)}{\zeta(1-6\sigma_0-6it)}\right|\\
			\nonumber &&\ll |t|^{\frac {1}{2}+\sigma_0}
			\zeta (1-2\sigma_0) \zeta (1-3\sigma_0) \zeta (1-6\sigma_0) \\
			\nonumber &&\ll_{\sigma_0} 
			|t|^{\frac {1}{2}+\sigma_0}.
		\end{eqnarray}
		Thus we find
		\begin{eqnarray*}
			\nonumber	\frac {1}{2\pi} \int\limits_{10}^{T}{\left|\frac{\zeta(2\sigma_0+2it) \zeta (3\sigma_0+3it)}{\zeta(6\sigma_0+6it)}\right|\frac{x^{\sigma_0}}{\left|\sigma_0+it\right|}dt}
			&&\ll \int\limits_{10}^{T}\frac{t^{\frac {1}{2}+\sigma_0} x^{\sigma_0}}{t}dt\\
			&&\ll_{\sigma_0}  x^{\sigma_0} \ T^{\frac {1}{2}+\sigma_0}
		\end{eqnarray*}
		since $ -\frac {1}{2} < \sigma_0 < 0$.

		Hence we obtain
		\begin{equation}\label{E21}
			E_2\ll x^{\sigma_0} + x^{\sigma_0} \ T^{\frac {1}{2}+\sigma_0} \ll_{\sigma_0} x^{\sigma_0} \ T^{\frac {1}{2}+\sigma_0}.
		\end{equation}

		\subsection{Horizontal line contribution in absolute value ($E_1$)}\mbox{} \vspace{2em}\\
		Now we have to compute the the contribution from the horizontal lines in absolute value. Here we split the range of the interval into six pieces  $[\sigma_0,0) $, $[0,\frac{1}{12})$, $[\frac{1}{12}, \frac{1}{6})$, $[\frac{1}{6}, \frac{1}{4})$, $[\frac{1}{4}, \frac{1}{3})$   
		and $[\frac{1}{3},\frac {1}{2}+\frac{1}{\log x}]$. Writing $F_2 (s) := \frac {\zeta (2s) \zeta (3s)}{\zeta (6s)}$ below, we find that :
		\begin{eqnarray*}\label{E22}
			E_1 &&\ll \int_{\sigma_0}^{\frac {1}{2}+\frac {1}{\log x}}  |F_2(\sigma + iT)| \left | \frac {x^{\sigma+iT}}{\sigma+iT} \right | d\sigma \\
			&&\ll \left \{ \int_{\sigma_0}^{0} + \int_{0}^{1/12} +\int_{1/12}^{1/6} +\int_{1/6}^{1/4} +\int_{1/4}^{1/3} +\int_{1/3}^{\frac {1}{2}+\frac {1}{\log x}} 
			\right \} | F_2(\sigma + iT)| \left | \frac {x^{\sigma+iT}}{\sigma+iT} \right | d\sigma \\
			&&\ll H_1 + H_2 + H_3 + H_4 + H_5 + H_6 \ \ ({\rm say}).
		\end{eqnarray*}
		{{\bf Case 1 :   }$\sigma_0 \le \sigma \le 0 $}\mbox{}\\
		Note that $-\frac {1}{2} < \sigma_0 < 0$. Hence,
		$1-2\sigma_0 \ge 1-2\sigma \ge 1, \ 1-3\sigma_0 \ge 1-3\sigma \ge 1, \ 1-6\sigma_0 \ge 1-6\sigma \ge 1$. 
		We use functional equation of the Riemann zeta function, bound for $|\chi(s)|$  and the bound in Lemma \ref{VL}. We find that
		\begin{equation}\label{E23}
			\left| F_2 (\sigma + iT) \right| \ll T^{\frac {1}{2} + \sigma} \ (\log T)^{2 + 3\epsilon}.
		\end{equation}
		Thus we conclude that
		\begin{eqnarray*}\label{E24}\nonumber
			H_1 &&= 	\int_{\sigma_0}^{0} | F_2(\sigma + iT)| \left | \frac {x^{\sigma+iT}}{\sigma+iT} \right | d \sigma \ll 	\int_{\sigma_0}^{0}  T^{\frac {1}{2} + \sigma} \ (\log T)^{2 + 3\epsilon} \frac {x^{\sigma}}{T} d \sigma
			\ll \frac {(\log T)^{3}}{T^{1/2}} \\
			&&\ll
			\begin{cases}
				1  \ \ {\rm if} \  \epsilon \le \kappa, \\
				1 \  {\rm if} \ \kappa = \epsilon
			\end{cases}
		\end{eqnarray*}

		{\bf Case 2 :} For $0\leqslant \sigma\leqslant \frac{1}{12}$, Note that $1\geqslant 1-2\sigma  \ge \frac{5}{6}$, $1\geqslant 1-3\sigma  \ge \frac{3}{4}$, $1\geqslant 1-6\sigma  \ge \frac{1}{2}$.

		We use the functional equation of $\zeta(s)$ and the uniform estimates in Lemma \ref{UL} and the estimate in Lemma \ref{RS} for the factor $\frac {1}{\zeta (1-6\sigma-6iT)}$ so that we get 
		
		\begin{eqnarray*}\label{E25}
			H_2 &&= 	\int_{0}^{1/12} | F_2(\sigma + iT)| \left | \frac {x^{\sigma+iT}}{\sigma+iT} \right | 
			d \sigma \\
			&&\ll 	\int_{0}^{1/12}  T^{\frac {1}{2} + \sigma} \  T^{2\kappa(1-(1-2\sigma))} (\log T) \	 T^{2\kappa(1-(1-3\sigma))} (\log T) \ \exp \left ( c (\log \log T)^2 \right ) \ 
			\frac {x^{\sigma}}{T} \ d\sigma \\ 
			&&\ll	\int_{0}^{1/12}  T^{\frac {1}{2} + \sigma} \ T^{10\kappa \sigma}
			\ \exp \left ( c_1^{\prime} (\log \log T)^2 \right ) \ 
			\frac {x^{\sigma}}{T} \ d\sigma \\ 	
			&&\ll \frac {T^{\frac {1}{2}+ \frac {10\kappa +1}{12}}}{T} \ \exp \left ( c_1^{\prime} (\log \log T)^2 \right ) \ x^{\frac {1}{12}} \\
			&&\ll
			\begin{cases}
				\frac {T^{\frac {1}{2}+ \frac {10\kappa +1}{12}}}{T} \ \exp \left ( c_1^{\prime} (\log \log T)^2 \right ) \ x^{\frac {1}{12}}  \ \ {\rm if} \  \epsilon \le \kappa, \\		  
				\frac {x^{\frac {1}{12}}  \ T^{\epsilon} \ \exp (c_1^{\prime} (\log \log T)^2)}{T^{5/12}} \ \ {\rm if} \  \kappa = \epsilon.
			\end{cases}
		\end{eqnarray*}
		
		{\bf Case 3 :} For $\frac {1}{12} \leqslant \sigma\leqslant \frac{1}{6}$, Note that $\frac {5}{6} \geqslant 1-2\sigma  \ge \frac{2}{3}$, $\frac {3}{4} \geqslant 1-3\sigma  \ge \frac{1}{2}$, $1\geqslant 6\sigma  \ge \frac{1}{2}$.

		We use the functional equation of $\zeta(s)$ and the uniform estimates in Lemma \ref{UL} for the factors $\zeta (2\sigma+2iT)$ and $\zeta (3\sigma+3iT)$, and the estimate in Lemma \ref{RS} for the factor $\frac {1}{\zeta (6\sigma+6iT)}$ so that we get 
		
		\begin{eqnarray*}\label{E25}
			H_3 &&= 	\int_{1/12}^{1/6} | F_2(\sigma + iT)| \left | \frac {x^{\sigma+iT}}{\sigma+iT} \right |	d \sigma \\
			&&\ll 	\int_{1/12}^{1/6}  T^{1-5 \sigma} \  T^{2\kappa(1-(1-2\sigma))} (\log T) \	 T^{2\kappa(1-(1-3\sigma))} (\log T) \ \exp \left ( c (\log \log T)^2 \right ) \ 
			\frac {x^{\sigma}}{T} \ d\sigma \\ 
			&&\ll	\int_{1/12}^{1/6}   T^{1+ (10\kappa - 5)\sigma}
			\ \exp \left ( c_2^{\prime} (\log \log T)^2 \right ) \ 
			\frac {x^{\sigma}}{T} \ d\sigma \\ 	
			&&\ll T^{\frac {10\kappa - 5}{12}} \ \exp \left ( c_2^{\prime} (\log \log T)^2 \right ) \ x^{\frac {1}{6}} \\
			&&\ll 
			\begin{cases}
				T^{\frac {10\kappa - 5}{12}} \ \exp \left ( c_2^{\prime} (\log \log T)^2 \right ) \ x^{\frac {1}{6}} \ \ {\rm if} \  \epsilon \le \kappa, \\
				\frac {x^{\frac {1}{6}} \ T^{\epsilon} \ \exp (c_2^{\prime} (\log \log T)^2)}{T^{5/12}}  \ \ {\rm if} \  \kappa = \epsilon.
			\end{cases}
		\end{eqnarray*}
		Note that we have used $10\kappa - 5 < 0$.

		{\bf Case 4 :} For $\frac {1}{6} \leqslant \sigma\leqslant \frac{1}{4}$, Note that $\frac {2}{3} \geqslant 1-2\sigma  \ge \frac{1}{2}$, $\frac {3}{4} \geqslant 3\sigma  \ge \frac{1}{2}$, $3/2 \geqslant 6\sigma  \ge 1$.

		We use the functional equation of $\zeta(s)$ and the uniform estimates in Lemma \ref{UL} for the factor $\zeta (2\sigma+2iT)$, only the uniform estimate for $\zeta (3\sigma+3iT)$, and the estimate in  Vinogradov's Lemma \ref{VL} for the factor $\frac {1}{\zeta (6\sigma+6iT)}$ so that we get 
		
		\begin{eqnarray*}\label{E25}
			H_4 &&= 	\int_{1/6}^{1/4} | F_2(\sigma + iT)| \left | \frac {x^{\sigma+iT}}{\sigma+iT} \right |	d \sigma \\
			&&\ll 	\int_{1/6}^{1/4}  T^{\frac {1}{2}-2 \sigma} \  T^{2\kappa(1-(1-2\sigma))} (\log T) \	 T^{2\kappa(1-3\sigma)} (\log T) \ ( \log T )^{\frac {2}{3} + \epsilon} \ 
			\frac {x^{\sigma}}{T} \ d\sigma \\ 
			&&\ll	\int_{1/6}^{1/4}   T^{\frac {1}{2}+ 2\kappa - 2\sigma - 2\kappa \sigma}
			\ (\log T)^{\frac {8}{3}+\epsilon}  \ 
			\frac {x^{\sigma}}{T} \ d\sigma \\ 	
			&&\ll  \frac {T^{\frac {1}{6}+ \frac {5\kappa}{3}} (\log T)^{\frac {8}{3} + \epsilon}}{T}  \ x^{\frac {1}{4}} \\
			&&\ll 
			\begin{cases}
				\frac {T^{\frac {1}{6}+ \frac {5\kappa}{3}} (\log T)^{\frac {8}{3} + \epsilon}}{T}  \ x^{\frac {1}{4}}	\ \ {\rm if} \  \epsilon \le \kappa, \\
				\frac {x^{\frac {1}{4}} \ T^{2\epsilon} \ (\log T)^{\frac {8}{3}+\epsilon}}{T^{\frac {5}{6}}}  \ \ {\rm if} \  \kappa = \epsilon .
			\end{cases}
		\end{eqnarray*}

		{\bf Case 5 :} For $\frac {1}{4} \leqslant \sigma\leqslant \frac{1}{3}$, Note that $\frac {2}{3} \geqslant 2\sigma  \ge \frac{1}{2}$, $1 \geqslant 3\sigma  \ge \frac{3}{4}$, $2 \geqslant 6\sigma  \ge 3/2$.

		We use only the uniform estimate in Lemma \ref{UL} for $\zeta (2\sigma+2iT)$ and $\zeta (3\sigma+3iT)$, and also notice that the factor $\frac {1}{\zeta (6\sigma+6iT)} \ll 1$ so that we get (with $\kappa = \frac {1}{6}$)
		
		\begin{eqnarray*}\label{E25}
			H_5 &&= 	\int_{1/4}^{1/3} | F_2(\sigma + iT)| \left | \frac {x^{\sigma+iT}}{\sigma+iT} \right |	d \sigma \\
			&&\ll 	\int_{1/4}^{1/3}   T^{2\kappa(1-2\sigma)} (\log T) \	 T^{2\kappa(1-3\sigma)} (\log T) \ 
			\frac {x^{\sigma}}{T} \ d\sigma \\ 
			&&\ll	\int_{1/4}^{1/3}   T^{ 2\kappa (2 - 5\sigma )}  \ (\log T)^2 \
			\frac {x^{\sigma}}{T} \ d\sigma \\ 	
			&&\ll  T^{\frac {3}{2} \kappa -1} \ (\log T)^2  \ x^{\frac {1}{3}} \\
			&&\ll 
			\begin{cases}
				T^{\frac {3}{2} \kappa -1} \ (\log T)^2  \ x^{\frac {1}{3}}	\ \ {\rm if} \ \epsilon \le  \kappa, \\
				\frac {x^{\frac {1}{3}} \ T^{2\epsilon} \ (\log T)^{2}}{T}  \ \ {\rm if} \  \kappa = \epsilon.
			\end{cases}
		\end{eqnarray*}

		{\bf Case 6 :} For $\frac {1}{3} \leqslant \sigma\leqslant \frac{1}{2}+\frac {1}{\log x}$, Note that $1 + \frac {2}{\log x} \geqslant 2\sigma  \ge \frac{2}{3}$, $\frac {3}{2} + \frac {3}{\log x} \geqslant 3\sigma  \ge 1$, $3 + \frac {6}{\log x} \geqslant 6\sigma  \ge 2$.

		We use the uniform estimate in Lemma \ref{UL} for $\zeta (2\sigma+2iT)$, Vinogradov's  estimate in Lemma \ref{VL} for
		$\zeta 3\sigma+3iT)$, and also notice that the factor $\frac {1}{\zeta (6\sigma+6iT)} \ll 1$ so that we get (with $\kappa = \frac {1}{6}$)
		
		\begin{eqnarray*}\label{E25}
			H_6 &&= 	\int_{1/3}^{\frac {1}{2}+\frac {1}{\log x}} | F_2(\sigma + iT)| \left | \frac {x^{\sigma+iT}}{\sigma+iT} \right |	d \sigma \\
			&&\ll 	\int_{1/3}^{\frac {1}{2}+\frac {1}{\log x}}   T^{2\kappa(1-2\sigma)} (\log T) \	 (\log T)^{\frac {2}{3}+\epsilon} \ 
			\frac {x^{\sigma}}{T} \ d\sigma \\ 
			&&\ll \frac {T^{2\kappa (1-\frac {2}{3})} \ (\log T)^{\frac {5}{3}+\epsilon}}{T}  \ x^{\frac {1}{2}} \\
			&&\ll 
			\begin{cases}
				\frac {T^{2\kappa (1-\frac {2}{3})} \ (\log T)^{\frac {5}{3}+\epsilon}}{T}  \ x^{\frac {1}{2}}
				\ \ {\rm if} \   \epsilon \le \kappa, \\
				\frac {x^{\frac {1}{2}} \ T^{\epsilon} \  (\log T)^{\frac {5}{3}+\epsilon}}{T} \ \ {\rm if} \  \kappa = \epsilon.
			\end{cases}
		\end{eqnarray*}
		We also have the tail portion error :
		\begin{equation}
			E_3 \ll \frac {x^{\frac {1}{2} + \frac {1}{\log x}} (\log x)^2}{T}
		\end{equation}
		and the condition that $10 \le T \le x $.
		
		We first choose $T=c_1 x^{\frac {1}{2}}$ where $c_1 > 0$ (may be small but fixed). Since we assume the  {\it Strong Riemann Hypothesis}, we can very well take $\kappa = \epsilon$ for any $\epsilon > 0$. We now observe that
		\begin{equation*}
			E_1 \ll	H_1+H_2+H_3+H_4+H_5+H_6 \ll_{\epsilon} \frac {x^{\frac {1}{2}} \ T^{3\epsilon}}{T} \ll x^{3\epsilon}.
		\end{equation*} 
		Also, we note that $E_3 \ll_{\epsilon} x^{\epsilon}$.
		
		Now, we choose $\sigma_0 =  - \frac {1}{6} +\epsilon$ so that (for this choice of $T=c_1 x^{\frac {1}{2}}$),
		
		\begin{equation*}
			E_2 \ll x^{\sigma_0} \ T^{\frac {1}{2} + \sigma_0} 
			\asymp  x^{- \frac {1}{6} +\epsilon} \ T^{\frac {1}{3} + \epsilon}  
			\ll x^{2\epsilon}.
		\end{equation*}
		Note that $\sigma_0 = - \frac {1}{6} + \epsilon$ which serves our purpose to establish \ref{E5}. Also note that this value of $\sigma_0$ satisfies our required  condition that $-\frac {1}{3} < \sigma_0 < 0$. Thus with this choices of $T$ and $\sigma_0$, we find that
		
		\begin{equation*}
			E_1 + E_2 + E_3 \ll_{\epsilon} x^{10 \epsilon}.
		\end{equation*}
		This proves \ref{E5} of the theorem \ref{T1}. 
		
		To prove \ref{E5a}, we choose $T=c_2 x^{\frac {5}{12}}$ where $c_2 > 0$ (may be small but fixed). Since we assume the  {\it Strong Riemann Hypothesis}, we can very well take $\kappa = \epsilon$ for any $\epsilon > 0$. As before, we observe that
		\begin{equation*}
			E_1 \ll	H_1+H_2+H_3+H_4+H_5+H_6 \ll_{\epsilon} \frac {x^{\frac {1}{2}} \ T^{3\epsilon}}{T} \ll x^{\frac {1}{12}+5\epsilon}.
		\end{equation*} 
		Also, we note that $E_3 \ll_{\epsilon} x^{\frac {1}{12}+\epsilon}$.
		
		Now, we choose $\sigma_0 =  - \frac {3}{34} + \epsilon$ so that (for this choice of $T=c_2 x^{\frac {5}{12}}$),
		
		\begin{equation*}
			E_2 \ll x^{\sigma_0} \ T^{\frac {1}{2} + \sigma_0} 
			\asymp  x^{ - \frac {3}{34} + \epsilon } \ T^{\frac {7}{17}+\epsilon} 
			\ll x^{\frac {35}{204} - \frac {3}{34} + 3 \epsilon}
			\ll x^{\frac {1}{12} + 3 \epsilon}.
		\end{equation*}
		Note that $\sigma_0 = - \frac {3}{34} + \epsilon$ which serves our purpose to establish \ref{E5a}. Also note that this value of $\sigma_0$ satisfies our required  condition that $-\frac {1}{3} < \sigma_0 < 0$. Thus with this choices of $T$ and $\sigma_0$, we find that
		
		\begin{equation*}
			E_1 + E_2 + E_3 \ll_{\epsilon} x^{\frac {1}{12}+ 10 \epsilon}.
		\end{equation*}
		This proves \ref{E5a} of the theorem \ref{T1}. 
		
		\noindent
		{\bf Sketch of the proof of theorem \ref{T2} :}
		
		\noindent
		To prove \ref{E5b}, we choose $T=c_3 x^{\theta}$ with
		\[
		\theta := \frac {5}{12} \ \frac {1}{1-\frac {2\kappa}{3}} = \frac {5}{4} \ \frac {1}{3-2\kappa} 
		\]
		where $c_3 > 0$ (may be small but fixed). Note that we can very well here take $\kappa = \frac {13}{84}+\epsilon$. We note that $\theta < \frac {1}{2}$ as long as $\kappa < \frac {1}{4}$ and this has been more than true.

		As before, we observe that if $T \ll x^{\frac {1}{2}}$, then
		\begin{equation*}
			H_1 + H_2 + H_4 + H_5 + H_6 \ll 1 + \frac {T^{2\kappa (1-\frac {2}{3})} \ (\log T)^{\frac {5}{3}+\epsilon}}{T}  \ x^{\frac {1}{2}}
			\ll 1 + \frac {T^{\frac {2\kappa}{3}} \ (\log T)^{\frac {5}{3}+\epsilon}}{T}  \ x^{\frac {1}{2}}
		\end{equation*}
		and
		\begin{equation*} 
			H_3 \ll T^{\frac {10\kappa - 5}{12}} \ \exp \left ( c_2^{\prime} (\log \log T)^2 \right ) \ x^{\frac {1}{6}} 
			\ll \frac { T^{\frac {10\kappa +7}{12}}}{T} \ \exp \left ( c_2^{\prime} (\log \log T)^2 \right ) \ x^{\frac {1}{6}}. 
		\end{equation*} 
		Thus we observe that as long as $T \ll \min \left ( x^{\frac {1}{2}}, x^{\frac {4}{2\kappa+7}} \right ) = x^{\frac {1}{2}}$ (since $\kappa \le \frac {1}{2}$), we have
		\begin{eqnarray*}
			H_1 + H_2 + H_3 + H_4 + H_5 + H_6 &&\ll_{\epsilon} 1 + \frac {T^{\frac {2\kappa}{3}} \ T^{\epsilon}}{T}  \ x^{\frac {1}{2}} \\
			&&\ll 1 + T^{\frac {2\kappa - 3}{3}} \ T^{\epsilon} \ x^{\frac {1}{2}} \\
			&&\ll 1 + x^{\frac {1}{2} + \theta (\frac {2\kappa-3}{3}) + \epsilon} \\
			&&\ll_{\epsilon} x^{\frac {1}{12}+ 5\epsilon}.	
		\end{eqnarray*}
		Taking $\kappa = \frac {13}{84}+\epsilon$,
		we note that $E_3 \ll_{\epsilon} x^{\frac {1}{2}-\theta +\epsilon} \ll x^{\frac {1}{2}-\frac {105}{226} + \epsilon} \ll x^{\frac {4}{113}+\epsilon}$.
		
		Now, we choose $\sigma_0 = -\frac {1}{2}+\epsilon$ so that (for this choice of $T=c_3 x^{\theta}$ and $\theta=\frac {105}{226}$), we get
		
		\begin{equation*}
			E_2 \ll x^{\sigma_0} \ T^{\frac {1}{2} + \sigma_0} 
			\ll x^{ \epsilon}.
		\end{equation*}
		Thus with these choices of $T$ and $\sigma_0$, we find that
		
		\begin{equation*}
			E_1 + E_2 + E_3 \ll_{\epsilon} x^{\frac {1}{12}+ 10 \epsilon}
		\end{equation*}
		unconditionally.
		This proves \ref{E5b} of the theorem \ref{T2}. 
		\vskip 3mm
		\noindent
		\begin{rem}
			In proving \ref{E5} and \ref{E5a} of the theorem \ref{T1}, one may very well take $\sigma_0 = - \frac {1}{2}+\epsilon$ and argue. We have chosen different values of $\sigma_0$ just to emphazise that they also serve our purpose.
		\end{rem}

		\section{ On the general case of $k$-full numbers}
		
		Let $k \ge 3$ be any integer.
		Let $f_k(n)$ denote the characteristic function on the set of $k$-full numbers, namely $f_k(n) = 1$ if $n$ is $k$-full and  $0$ otherwise. We observe here  that (see page \cite{IV}) the generating function attached to this sequence of numbers $f_k(n)$ is indeed $F_k(s) =  H_k(s) \ G_k(s)$ for $\Re s > \frac {1}{k}$ where
		\begin{equation*}
			H_k(s)= \sum_{n=1}^{\infty} \frac {h_k(n)}{n^s} = \prod_{j=k}^{2k-1} \zeta (js) \
			{\rm for} \ \Re s > \frac {1}{k}
		\end{equation*}
		and
		\begin{equation*}
			G_k(s)= \sum_{n=1}^{\infty} \frac {g_k(n)}{n^s} = \frac {\Phi_k (s)}{\zeta ((2k+2)s)} \
			{\rm for} \ \Re s > \frac {1}{2k+2}.
		\end{equation*}
		More generally one expects an asymptotic formula of the kind
		\begin{equation*}
			B_k(x) = \sum_{n\leqslant x} f_k(n) = \sum_{j=0}^{k-1} c_{j,k} x^{\frac {1}{k+j}} +	E_k(x)
		\end{equation*}
		where $c_{j,k}$ is the residue at $s=\frac {1}{k+j}$ of the function $\frac {F_k(s)}{s}$
		for $j=0,1,2,\cdots, k-1$,
		and
		$E_k(x) (=  o(x^{\frac {1}{2k-1}}))$ is the error term whose tight order to be determined.
		For any fixed integer $k \ge 3$, let us define
		\begin{equation*}
			\eta_k :=  \ \inf  \ \{ \eta (> 0) : E_k (x) \ll x^{\eta}  \}.
		\end{equation*}
		
		Then,  unconditionally we know that (see pages 411-413 of \cite{IV})
		\begin{equation}\label{E100}
			\eta_3 \le \frac {577}{4176} = 0.1387 \dots \ \ ; \ \ \eta_k \le \frac {1}{2k} \ {\rm for} \ 4 \le k \le 10,
		\end{equation}
		and it is also mentioned that on the assumption of the Lindel\"of Hypothesis (LH), it follows that
		
		\begin{equation}\label{E101}
			\eta_k \le \frac {1}{2k}  \ \ {\rm for \ all \ integers} \ k \ge 11.
		\end{equation}
		We recall here that the Lindel\"of Hypothesis (LH) asserts that the bound
		\begin{equation*}
			\zeta (\sigma + it) \ll (|t| + 10)^{\epsilon}
		\end{equation*}
		holds for every $\sigma \ge \frac {1}{2}$ and for every $\epsilon > 0$.
		\vskip 2mm
		\noindent
		We assume below that $x\geqslant x_0$, where $x_0$ is sufficiently large.
		The aim of this section is to establish that :
		\vskip 3mm
		\noindent
		\begin{thm}\label{T4}
			For any fixed positive integer $k \ge 4$, the Lindel\"of Hypothesis (LH) implies the inequality
			\begin{equation}\label{E102}
				\eta_k \le \frac {k+1}{k(2k+3)}.
			\end{equation} 
		\end{thm}
		\begin{rem}
			We note that for any fixed integer $k \ge 4$, the inequalities
			$\frac {1}{2k+1} < \frac {k+1}{k(2k+3)} < \frac {1}{2k}$ hold. Hence Theorem \ref{T4} improves inequalities \ref{E100} and \ref{E101} for $k \ge 4$ under the assumption of the Lindelof Hypothesis.
		\end{rem}
		\vskip 3mm
		\noindent
		{\bf Proof of the theorem \ref{T4}:}  In the Perron's formula, we take the L-function here to be $F_k(s) = H_k(s) G_k(s)$. Thus, we have,
		\begin{equation*}
			B_k(x) := \sum \limits_{n \le x} f_k(n)  = \frac {1}{2\pi i} \int_{\frac {1}{k}+\epsilon-iT}^{\frac {1}{k}+\epsilon+iT}  F_k(s) \frac {x^s}{s} \ ds + O \left ( \frac {x^{\frac {1}{k}+\epsilon}}{T} \right ).
		\end{equation*}
		Now
		we move the line of integration to $\Re s = \frac {1}{2k+2} + \epsilon = \sigma_1$ (say) and the main terms are coming from the simple poles at $s=\frac {1}{k+j}$ \ (for $j=0,1,2,\cdots, k-1$). 
		We notice that $G_k(s) \ll 1$ for $\Re s \ge \sigma_1$ and it is a harmless factor. We also observe that
		\[
		(k+j) \sigma_1 = \frac {k+j}{2k+2} + (k+j)\epsilon \ge \frac {k+1}{2(k+1)} = \frac {1}{2}
		\]
		for $j=1,2,\cdots,k-1$ and
		\[
		k\sigma_1 = \frac {k}{2(k+1)} + k\epsilon < \frac {1}{2}
		\]
		for any small positive $\epsilon$ (for instance we can assume that $\epsilon < \frac {1}{2k(k+1)}$).
		This enables us to use the functional equation to $\zeta(ks)$ to take out the growth of the conversion factor and the uniform estimate of the Riemann zeta function for the factor $\zeta (1-ks)$ (under LH) of $H_k(s)$ and for rest of the factors $\zeta ((k+j)s)$ (for $j=1,2,\cdots, k-1$), we only use the uniform estimate (under LH) as $(k+j)\sigma_1 \ge \frac {1}{2}$.
		
		The left vertical line contributes in absolute value at most :
		
		\begin{eqnarray*}\label{E103}
			\frac {1}{2\pi} \int_{\sigma_1-iT}^{\sigma_1+iT} |F_k (\sigma_1+it)| \  \left | \frac {x^{\sigma_1+it}}{\sigma_1+it}
			\right | dt
			&&\ll  x^{\sigma_1} + \int_{\Re s = \sigma_1,10 \le |t|\le T}
			\left | F_k (s)	\frac {x^s}{s} \right | \ dt \\
			&&\ll x^{\sigma_1}  + x^{\sigma_1}\int_{10}^{T}  t^{\frac {1}{2}-k\sigma_1} \ . \ t^{\epsilon} \ . \ t^{(k-1)\epsilon} \ \frac {dt}{t} \\
			&&\ll  x^{\sigma_1}  + x^{\sigma_1} \  T^{\frac {1}{2}-k\sigma_1} \ . \ T^{(k+1)\epsilon} \\
			&&\ll x^{\sigma_1} \  T^{\frac {1}{2}-k\sigma_1} \ . \ T^{(k+1)\epsilon} \\
			&&\ll  x^{\sigma_1} \ T^{\frac {1}{2(k+1)}} \ T^{3k\epsilon}.
		\end{eqnarray*}
		Note that we have used the estimate under LH.
		
		To estimate  the contributions coming from the two horizontal (top and bottom)  lines, first
		we split the interval $\left [ \sigma_1, \frac {1}{k} + \epsilon \right ]$ into union of abutting sub-intervals, namely of : $[\sigma_1, \frac {1}{2k}], [\frac {1}{2k}, \frac {1}{k} + \epsilon ]$. 
		Let us denote the contributions in absolute value coming from these horizontal bits as  $H_1(k), H_2(k)$ respectively. 
		
		We find that :
		
		\begin{eqnarray*}
			H_1(k) &&\ll \int_{\sigma_1}^{\frac {1}{2k}} \left | F_k(\sigma+iT) \frac {x^{\sigma+iT}}{\sigma+iT} \right | \ d\sigma \\
			&&\ll \int_{\sigma_1}^{\frac {1}{2k}} T^{\frac {1}{2}-k\sigma} \ . \ T^{\epsilon} \ . \  T^{(k-1)\epsilon} \ \frac {x^{\sigma}}{T} \ d\sigma \\
			&&\ll_k T^{\frac {1}{2}-k\sigma_1} \ T^{k\epsilon} \ \frac {x^{\frac {1}{2k}}}{T} \\
			&&\ll T^{\frac {1}{2k}} \ T^{k\epsilon} \ \frac {x^{\frac {1}{2k}}}{T} \\
			&&\ll \frac {x^{\frac {1}{k}}}{T} \ T^{k\epsilon},
		\end{eqnarray*}
		since $\frac {1}{2}-k\sigma_1 < \frac {1}{2} - \frac {k}{2(k+1)} = \frac {1}{2(k+1)} < \frac {1}{2k}$ and $10 \le T \le x$.
		
		\begin{eqnarray*}
			H_2(k) &&\ll  \int_{\frac {1}{2k}}^{\frac {1}{k}+\epsilon} \left | F_k(\sigma+iT) \frac {x^{\sigma+iT}}{\sigma+iT} \right | \ d\sigma \\
			&&\ll  \int_{\frac {1}{2k}}^{\frac {1}{k}+\epsilon} T^{k\epsilon} \ \frac {x^{\sigma}}{T} \ d\sigma \\
			&&\ll \frac {x^{\frac {1}{k}+\epsilon} T^{k\epsilon}}{T}.
		\end{eqnarray*}
		Already we have an error $\ll \frac {x^{\frac{1}{k}+\epsilon}}{T}$ coming from the tail portion of the Perron's formula. Collecting all these errors together, we find that
		\begin{equation*}
			E_k(x) \ll \frac {x^{\frac {1}{k}+\epsilon}}{T} \ T^{k\epsilon} +
			x^{\sigma_1} \ T^{\frac {1}{2(k+1)}} \ T^{3k\epsilon} \ll
			\frac {x^{\frac {1}{k}+\epsilon}}{T} \ T^{k\epsilon} + x^{\frac {1}{2(k+1)}+\epsilon} T^{\frac {1}{2(k+1)}} \ T^{3k\epsilon}.
		\end{equation*}
		Now, we choose $T$ such that $\frac {x^{\frac {1}{k}}}{T} \asymp x^{\frac {1}{2(k+1)}} \ T^{\frac {1}{2(k+1)}}$ so that our $T \asymp x^{\frac {k+2}{k(2k+3)}}$ and thus we obtain that
		\begin{equation*}
			E_k(x)\ll_k x^{\frac {k+1}{k(2k+3)}+10k\epsilon}.
		\end{equation*}
		This implies that $\eta_k \le \frac {k+1}{k(2k+3)}$. This proves the theorem \ref{T4}.

	\end{section}
	
	{\large \bf Acknowledgements:} The  first author  wishes to thank the Kerala Infrastructure Investment Fund Board (KIIFB) for its financial support to him for being a Visiting Professor at the  Kerala School of Mathematics (KSoM), kozhikode, Kerala, India.

\end{document}